\documentclass[reqno]{amsart}
\usepackage{hyperref}
\usepackage{bbm}
\usepackage{amsfonts,amsmath,amsxtra,amsthm,amssymb,latexsym,enumerate}
\usepackage{ifpdf}
\usepackage{tikz}
\usetikzlibrary{patterns}
\usetikzlibrary{decorations.pathreplacing}

\usepackage{color}

\newcommand*{\mailto}[1]{\href{mailto:#1}{\nolinkurl{#1}}}
\newcommand{\arxiv}[1]{\href{http://arxiv.org/abs/#1}{arXiv:#1}}

\newcommand{\msc}[1]{\href{http://www.ams.org/msc/msc2010.html?t=&s=#1}{#1}}

\newcommand{\ack}{\section*{Acknowledgments}}

\newtheorem{theorem}{Theorem}[section]
\newtheorem{corollary}[theorem]{Corollary}
\newtheorem{proposition}[theorem]{Proposition}
\newtheorem{lemma}[theorem]{Lemma}
\newtheorem{hypothesis}{Hypothesis}[section]
\theoremstyle{definition}
\newtheorem{definition}[theorem]{Definition}
\newtheorem{example}[theorem]{Example}
\newtheorem{remark}[theorem]{Remark}

\newcommand{\be}{\begin{equation}}
\newcommand{\ee}{\end{equation}}
\newcommand{\wh}{\widehat}
\newcommand{\id}{{\mathbbm 1}}

\numberwithin{equation}{section}

\DeclareMathOperator{\diam}{diam} 
\DeclareMathOperator{\Span}{span} 
\DeclareMathOperator{\dom}{dom}
\DeclareMathOperator{\sym}{sym}

\def\mG{\mathsf{G}} 
  
\newcommand\R{{\mathbb{R}}}
\newcommand\C{{\mathbb{C}}}
\newcommand\Z{{\mathbb{Z}}}

\newcommand{\gt}{\mathfrak{t}}
\newcommand{\gC}{\mathfrak{C}}

\newcommand\cI{{\mathcal{I}}}
\newcommand\cT{{\mathcal{T}}}
\newcommand\cL{{\mathcal{L}}}
\newcommand\cN{{\mathcal{N}}}
\newcommand\cU{{\mathcal{U}}}
\newcommand\cG{{\mathcal{G}}}
\newcommand\cE{{\mathcal{E}}}
\newcommand\cP{{\mathcal{P}}}
\newcommand\cV{{\mathcal{V}}}
\newcommand\cR{{\mathcal{R}}}
\newcommand{\varrhoo}{\varrho}

\newcommand\bH{{\mathbf{H}}}

\newcommand\rH{{\rm{H}}}
\newcommand\E{{\rm{e}}}
\newcommand\vol{{\rm{vol}}}

\newcommand\Nr{{\rm{n}}}
\newcommand\I{{\rm{i}}}
\newcommand\rD{{\rm{d}}}

\def\wt#1{{{\widetilde #1} }}

\begin{document}

\title[Gaffney Laplacian on Metric Graphs]{A Note on The Gaffney Laplacian\\ on Infinite Metric Graphs}

\dedicatory{To Mark Malamud, our teacher, colleague and friend, on the occasion of his 70th birthday}

\author[A. Kostenko]{Aleksey Kostenko}
\address{Faculty of Mathematics and Physics\\ University of Ljubljana\\ Jadranska ul.\ 21\\ 1000 Ljubljana\\ Slovenia\\ and 
Institute for Analysis and Scientific Computing\\ Vienna University of Technology\\ Wiedner Hauptstra\ss e 8-10/101\\1040 Vienna\\ Austria}
\email{\mailto{Aleksey.Kostenko@fmf.uni-lj.si}}
%\urladdr{\url{https://www.fmf.uni-lj.si/~kostenko/}}

\author[N. Nicolussi]{Noema Nicolussi}
\address{Faculty of Mathematics\\ University of Vienna\\
Oskar-Morgenstern-Platz 1\\ 1090 Vienna\\ Austria}
\email{\mailto{noema.nicolussi@univie.ac.at}}

\thanks{{\it Research supported by the Austrian Science Fund (FWF) 
under Grants No.\ P 28807 (A.K. and N.N.) and W 1245 (N.N.), and by the Slovenian Research Agency (ARRS) under Grant No.\ J1-1690 (A.K.)}}


\begin{abstract}
We show that the deficiency indices of the minimal Gaffney Laplacian on an infinite locally finite metric graph are equal to the number of finite volume graph ends. Moreover, we provide criteria, formulated in terms of finite volume graph ends, for the Gaffney Laplacian to be closed. 
\end{abstract}

\maketitle

%%%%%%%%%%%%%%%%%%%%%%%%%%%%%%%%%%%%%%%%%%%%%%%%%%%%%%%%%%%%
%%%%%%%%%%%%%%%%%%%%%%%%%%%%%%%%%%%%%%%%%%%%%%%%%%%%%%%%%%%%
\section{Introduction}
%%%%%%%%%%%%%%%%%%%%%%%%%%%%%%%%%%%%%%%%%%%%%%%%%%%%%%%%%%%%
%%%%%%%%%%%%%%%%%%%%%%%%%%%%%%%%%%%%%%%%%%%%%%%%%%%%%%%%%%%%

The standard way to associate an operator with the Laplacian on an unbounded domain in $\R^n$ or 
on a non-compact Riemannian manifold is to consider it either on smooth compactly supported functions (the so-called {\em pre-minimal operator}) or on the largest possible natural domain when Laplacian is understood in a distributional sense in $L^2$ 
 (the {\em maximal operator}). Then the pre-minimal Laplacian is essentially self-adjoint if and only if its $L^2$ closure,  the {\em minimal operator}, coincides with the maximal operator. For geodesically complete manifolds the essential self-adjointness was proved by W.\ Roelcke \cite{roe} (see also \cite{che,str}). It is impossible to give even a brief account on the subject and we only refer to  \cite{bms, ebe, shu}, some recent work in the case of non-complete manifolds \cite{ma99, ma05} and also for weighted graph Laplacians \cite{hklw, hkmw}. If we further restrict the maximal operator to functions having finite energy, the self-adjointness of this operator (called the {\em Gaffney Laplacian} \cite{gaf, gm}) is equivalent to the uniqueness of a Markovian extension.  Clearly, essential self-adjointness implies the uniqueness of Markovian extensions, but the converse is not necessarily true. We refer to \cite[Chapter I]{ebe} for an excellent account on importance and applications of both the self-adjoint and Markovian uniqueness.

The main object of our paper is a {\em quantum graph}, i.e., a Laplacian on a metric graph.
From the perspective of Dirichlet forms, quantum graphs play an important role as an intermediate setting between Laplacians on Riemannian manifolds and discrete Laplacians on weighted graphs. The most studied quantum graph operator is the \emph{Kirchhoff Laplacian}, which provides the analog of the Laplace--Beltrami operator in the setting of metric graphs. Whereas on finite metric graphs the Kirchhoff Laplacian is always self-adjoint, the question is more subtle for \emph{graphs with infinitely many edges}.  Geodesic completeness (w.r.t. the natural path metric) guarantees self-adjointness of the (minimal) Kirchhoff Laplacian, however, this result is far from being optimal \cite[\S 4]{ekmn}. 
The present paper is a complement to the recent work \cite{kmn19}, where a relationship between one of the classical notions of boundaries for infinite graphs, \emph{graph ends}, and self-adjoint extensions of the minimal Kirchhoff Laplacian on a metric graph was established. More precisely, the notion of \emph{finite volume} for ends of a metric graph introduced in \cite{kmn19} turns out to be the proper notion of a boundary for Markovian extensions of the Kirchhoff Laplacian. Our main goal is to elaborate on the relationship between finite volume graph ends and the Gaffney Laplacian, which is defined as the restriction of the maximal Kirchhoff Laplacian to functions having finite energy (i.e., $H^1$ functions). First of all, one of the main results of \cite{kmn19} provides a transparent geometric characterization of the self-adjointness of the Gaffney Laplacian: {\em the underlying metric graph has no finite volume ends} (see Lemma \ref{cor:Markov} below). Our first main result shows that the deficiency indices of the minimal Gaffney Laplacian (i.e., the degree of the unitary group parameterizing its self-adjoint extensions) are in fact equal to the number of finite volume graph ends (Theorem \ref{th:n-pm=C0}). The Gaffney Laplacian has several advantages comparing to the maximal Kirchhoff Laplacian, although one of the main disadvantages is the fact that it is not necessarily closed. Our second main result,  Theorem \ref{th:main} provides necessary and sufficient conditions for the Gaffney Laplacian to be closed (see also Prop.~\ref{prop:graphseq}). These conditions are stated in terms of finite volume graph ends and in certain cases of interest (graphs of finite total volume or Cayley graphs of countable finitely generated groups) they give rise to a transparent geometrical criterion: {\em the Gaffney Laplacian is closed if and only if the underlying metric graph has finitely many finite volume ends}, which is further equivalent to the fact that the deficiency indices of the minimal Gaffney Laplacian are finite. 
If the Gaffney Laplacian is not closed, then the most important question is how to describe its $L^2$ closure. It does not seem realistic to us to obtain a complete answer to this question and we demonstrate by examples that under certain symmetry conditions the closure of the Gaffney Laplacian may coincide with the maximal Kirchhoff Laplacian.  

Let us now briefly describe the structure of the article. Section \ref{sec:prelim} is of preliminary character where we collect basic notions and facts about graphs and metric graphs (Section \ref{ss:II.01}); graph ends (Section \ref{ss:II.02}); Sobolev spaces on metric graphs (Section \ref{ss:II.03}); Kirchhoff, Dirichlet and Neumann Laplacians on metric graphs (Section \ref{ss:II.04}). 
The main results of the present paper are collected in Section \ref{sec:Gaffney}. We introduce the minimal Gaffney Laplacian and the Gaffney Laplacian, study their properties and also investigate their relationship with the Kirchhoff, Dirichlet and Neumann Laplacians. 
In the final section we discuss several explicit examples. 
%%%%%%%%%%%%%%%%%%%%%%%%%%%%%%%%%%%%%%%%%%%%%%%%%%%%%%%%%%%%
\subsection*{Notation}
%%%%%%%%%%%%%%%%%%%%%%%%%%%%%%%%%%%%%%%%%%%%%%%%%%%%%%%%%%%%
$\Z$, $\R$, $\C$ have their usual meaning; $\Z_{\ge a} := \Z\cap [a,\infty)$.\\
$z^\ast$ denotes the complex conjugate of $z\in\C$.\\  
For a given set $S$, $\#S$ denotes its cardinality if $S$ is finite; otherwise we set $\#S=\infty$.\\
If it is not explicitly stated otherwise, we shall denote by $(x_n)$ a sequence $(x_n)_{n=0}^\infty$.

%%%%%%%%%%%%%%%%%%%%%%%%%%%%%%%%%%%%%%%%%%%%%%%%%%%%%%%%%%%%
%%%%%%%%%%%%%%%%%%%%%%%%%%%%%%%%%%%%%%%%%%%%%%%%%%%%%%%%%%%%
\section{Quantum graphs and graph ends} \label{sec:prelim}
%%%%%%%%%%%%%%%%%%%%%%%%%%%%%%%%%%%%%%%%%%%%%%%%%%%%%%%%%%%%
%%%%%%%%%%%%%%%%%%%%%%%%%%%%%%%%%%%%%%%%%%%%%%%%%%%%%%%%%%%%

%%%%%%%%%%%%%%%%%%%%%%%%%%%%%%%%%%%%%%%%%%%%%%%%%%%%%%%%%%%%
\subsection{Combinatorial and metric graphs}\label{ss:II.01}
%%%%%%%%%%%%%%%%%%%%%%%%%%%%%%%%%%%%%%%%%%%%%%%%%%%%%%%%%%%%
In what follows, $\cG_d = (\cV, \cE)$ will be an unoriented graph with countably infinite sets of vertices $\cV$ and edges $\cE$. For two vertices $u$, $v\in \mathcal{V}$ we shall write $u\sim v$ if there is an edge $e_{u,v}\in \mathcal{E}$ connecting $u$ with $v$. 
For every $v\in \mathcal{V}$, we denote  the set of edges incident to the vertex $v$ by $\cE_v$ and 
\be\label{eq:combdeg}
\deg(v):= \#\{e|\, e\in\cE_v\} 
\ee
is called \emph{the degree} of a vertex $v\in\cV$. 
\emph{A path} $\cP$ of length $n\in\Z_{\ge 0}\cup\{\infty\}$ is a sequence of vertices $(v_0,v_1,\dots, v_n)$ such that $v_{k-1}\sim v_k$ for all $k\in \{1,\dots,n\}$. 

The following assumption is imposed throughout the paper.

\begin{hypothesis}\label{hyp:locfin}
$\cG_d$ is \emph{simple} (no loops or multiple edges)\footnote{This assumption is not a restriction and we only need it to streamline the exposition.}, \emph{locally finite} ($\deg(v) < \infty$ for every $v \in \cV$) and \emph{connected} (for any $u,v\in\cV$  there is a path connecting $u$ and $v$).
\end{hypothesis}

Assigning to each edge $e \in \cE$ a finite length $|e|\in (0,\infty)$ turns $\cG_d$ into a \emph{metric graph} $\cG:=(\cV,\cE,|\cdot|) = (\cG_d,|\cdot|)$. The latter equips $\cG$ with a (natural) topology and metric. More specifically (see, e.g., \cite[Chapter 1.1]{hae}),  a metric graph $\cG$ is a Hausdorff topological space with countable base such that each point $x\in \cG$ has a neighbourhood $\cE_x(r)$ homeomorphic to a star-shaped set $\cE(\deg(x),r_x)$ of degree $\deg(x)\ge 1$,
\[
\cE(\deg(x),r_x) := \{z= r\E^{2\pi \I k/\deg(x)}|\, r\in [0,r_x),\ k=1,\dots,\deg(x)\}\subset \C.
\]
Identifying every edge $e\in\cE$ with a copy of an interval of length $|e|$ and also identifying the ends of the edges that correspond to the same vertex $v$, $\cG$ can be equipped with \emph{the natural path metric} $\varrhoo$ --- the distance between two points $x,y\in\cG$ is defined as the length of the ``shortest" path connecting $x$ and $y$. 

%%%%%%%%%%%%%%%%%%%%%%%%
\subsection{Graph ends}\label{ss:II.02}  
%%%%%%%%%%%%%%%%%%%%%%%%

 A sequence of distinct vertices $(v_n)$ such that $v_{n}\sim v_{n+1}$ for all $n\in\Z_{\ge 0}$ is called a \emph{ray}.  
Two rays $\cR_1,\cR_2$ are called \emph{equivalent} 
 if there is a third ray containing infinitely many vertices of both $\cR_1$ and $\cR_2$. 
 An equivalence class of rays is called a \emph{graph end of  $\cG_d$}.
 
Considering a metric graph $\cG$ as a topological space, one can introduce topological ends. 
Consider sequences $\cU = (U_n)$ of non-empty open connected subsets of $\cG$ with compact boundaries and such that $U_{n+1} \subseteq U_{n}$ for all $n\ge 0$ and $\bigcap_{n\ge 0} \overline{U_n} = \emptyset$. Two such sequences $\cU$ and $\cU'$ are called \emph{equivalent} if for all $n\ge0$ there exist $j$ and $k$ such that $U_n \supseteq U_j'$ and $U_n' \supseteq U_k$. 
An equivalence class $\gamma$ of sequences is called a \emph{topological end} of $\cG$ and $\gC(\cG)$ denotes the set of topological ends of $\cG$. There is a natural bijection between topological ends of a locally finite metric graph $\cG$ and graph ends of the underlying combinatorial graph $\cG_d$: for every topological end $\gamma\in\gC(\cG)$ of $\cG$ there exists a unique graph end $\omega_{\gamma}$ of $\cG_d$ such that for every sequence $\cU$ representing $\gamma$, each $U_n$ contains a ray from $\omega_{\gamma}$  (see \cite[\S 21]{woe}, \cite[\S 8.6 and also p.277--278]{die} for further details). 

One of the main features of graph ends is that they provide a rather refined way of compactifying graphs, called the \emph{end (or Freudenthal) compactification} of $\cG$ (see \cite[\S 8.6]{die}, \cite{woe} and also \cite[\S 2.2]{kmn19}).

\begin{definition}
An end $\omega$ of a graph $\cG_d$ is called  \emph{free} if there is a finite set $X$ of vertices such that $X$ separates $\omega$ from all other ends of the graph.
\end{definition}

\begin{remark} \label{rem:freetop} 
Notice that an end $\gamma \in \gC(\cG)$ is free exactly when 
 there exists a connected subgraph $\wt \cG$ with compact boundary $\partial \wt \cG$\footnote{For a subgraph $\wt\cG$ of $\cG$ its boundary is  $\partial\wt\cG = \{v\in \cV(\wt\cG)| \deg_{\wt\cG}(v) < \deg(v)\}$ and hence  $\partial \wt \cG$ is compact exactly when $\# \partial \wt \cG < \infty$.} such that 
$U_n \subseteq \wt \cG$ eventually for any sequence $ \cU = (U_n)$ representing $ \gamma$ and $U_n' \cap \wt \cG = \varnothing$ eventually for all sequences $\cU' = (U_n') $ representing an end $\gamma' \neq \gamma$.  
\end{remark}

We also need the following notion introduced in \cite{kmn19}.

\begin{definition} \label{def:finvol}
A topological end $\gamma \in \gC(\cG)$ has \emph{finite volume} if there is a sequence $\cU = (U_n)$  representing $\gamma$ such that $\vol(U_n) < \infty$\footnote{As usual, $\vol(A)$ denotes the Lebesgue measure of a measurable set $A \subseteq \cG$.} for some $n$. Otherwise $\gamma$ has \emph{infinite volume}. 
 The set of all finite volume ends is denoted by $\gC_0(\cG)$. 
\end{definition}

%%%%%%%%%%%%%%%%%%%%%%%%
\subsection{Function spaces on metric graphs}\label{ss:II.03}  
%%%%%%%%%%%%%%%%%%%%%%%%
Identifying every edge $e\in\cE$ with the copy of $\cI_e = [0,|e|]$ (and hence assigning an orientation on $\cG$), we can introduce Sobolev spaces on edges and on $\cG$.  First of all, the Hilbert space $L^2(\cG)$ of functions $f\colon \cG\to \C$ is defined by  
\[
L^2(\cG) = \bigoplus_{e\in\cE} L^2(e) = \Big\{f=\{f_e\}_{e\in\cE}\big|\, f_e\in L^2(e),\ \sum_{e\in\cE}\|f_e\|^2_{L^2(e)}<\infty\Big\}.
\]
The subspace of compactly supported $L^2(\cG)$ functions will be denoted by
\begin{equation*}
	L^2_c(\cG) = \big\{f \in L^2(\cG)| \; f \neq 0 \text{ only on finitely many edges } e \in \cE\big\}.
\end{equation*}
For edgewise locally absolutely continuous functions on $\cG$ let us denote by $\nabla$ the edgewise first derivative,
\be\label{eq:nabla}
\nabla f:= f'.
\ee
Then for every edge $e\in\cE$, 
\[
H^1(e) = \{f\in AC(e)|\, \nabla f\in L^2(e)\},\quad H^2(e) = \{f\in H^1(e)|\,  \nabla f\in H^1(e)\},
\]
where $AC(e)$ is the space of absolutely continuous functions on $e$.  
Next $H^n(\cG\setminus\cV)$, $n\in \{1,2\}$ is defined as the space of functions $f\colon \cG\to \C$ such that   
\[
H^n(\cG\setminus\cV) = \bigoplus_{e\in\cE} H^n(e) = \Big\{f=\{f_e\}_{e\in\cE}\big|\, f_e\in H^n(e),\ \sum_{e\in\cE}\|f_e\|^2_{H^n(e)}<\infty\Big\}.
\]
It becomes a Hilbert space when equipped with the norm $
\| f \|^2_{H^n} := %\| f \|^2_{L^2(\cG)} +  \| \nabla f \|^2_{L^2(\cG)} = 
\sum_{e\in\cE}\|f_e \|^2_{H^n(e)}$, where $\|\cdot\|^2_{H^n(e)} =\| \cdot\|_{L^2(e)}^2+ \|\nabla^n \cdot\|_{L^2(e)}^2$, $n\in\{1,2\}$.

The first Sobolev space on $\cG$ is defined by
\[
H^1(\cG) = H^1(\cG\setminus\cV)\cap C(\cG),
\]
which is also a Hilbert space when equipped with the above norm. 
We also define $H^1_0(\cG)$ as the closure in $H^1(\cG)$ of $H^1_c(\cG) = H^1(\cG)\cap L^2_c(\cG)$.

The Sobolev space $H^1 (\cG)$ is continuously embedded in $C_b(\cG) = C(\cG)\cap L^\infty(\cG)$ (see, e.g., \cite[Lemma~3.2]{kmn19}) and, moreover, every function $f\in H^1(\cG)$ admits a unique continuous extension to the end compactification $\wh{\cG}$ of $\cG$ (\cite[Prop.~3.5]{kmn19}): 
 for every $f \in H^1(\cG)$ and a (topological) end $\gamma \in \gC(\cG)$, we  define
\begin{equation} \label{eq:defvalueend}
	f( \gamma)  := \lim_{n\to \infty}f(v_n),
\end{equation}
where $\cR = (v_n) \in \omega_\gamma$ is any ray belonging to the corresponding graph end $\omega_\gamma$. 

It turns out that finite volume graph ends serve as a proper boundary for the Sobolev space $H^1(\cG)$. Namely, considering $ H^1(\cG)$ as a subalgebra of $C_b (\cG)$, it was proved in \cite[\S 3]{kmn19} that its closure is isomorphic to $C_0(\cG \cup \gC_0(\cG))$. In particular, ends having infinite volume lead to trivial values, that is, 
\[
f(\gamma) = 0
\]
for every $f\in H^1(\cG)$ if and only if $\gamma\notin \gC_0(\cG)$. Moreover, by \cite[Theorem~3.10]{kmn19}, 
\begin{equation} \label{eq:h10} 
	H^1_0 (\cG) =\{ f \in H^1(\cG) | \; f(\gamma) = 0 \text{ for all } \gamma \in \gC(\cG) \},
\end{equation}
and hence $H^1(\cG) = H^1_0(\cG)$ exactly when $\cG$ has no finite volume ends, $\gC_0(\cG) = \emptyset$.

%%%%%%%%%%%%%%%%%%%%%%%%%%%%%%%%%%%%%%%%%%%%%%%%%%%%%%%%%%%%
\subsection{Kirchhoff, Dirichlet and Neumann Laplacians}\label{ss:II.04}
%%%%%%%%%%%%%%%%%%%%%%%%%%%%%%%%%%%%%%%%%%%%%%%%%%%%%%%%%%%%

Let $\cG$ be a metric graph satisfying Hypothesis \ref{hyp:locfin}. In the Hilbert space $L^2(\cG)$, we can introduce the Laplacian $\Delta$ defined on the maximal domain $H^2(\cG\setminus\cV)$ and acting on each edge as the negative second derivative, $\Delta = -\nabla^2$. If $v$ is a vertex of the edge $e \in \cE$, then for every $f\in H^2(e)$ the quantities 
\begin{align}\label{eq:tr_fe}
f_e(v) & := \lim_{x_e\to v} f(x_e), &  \partial_e f(v) & := \lim_{x_e\to v} \frac{f(x_e) - f(v)}{|x_e - v|},
\end{align}
are well defined. The Kirchhoff (also called \emph{standard} or \emph{Kirchhoff--Neumann}) boundary conditions at every vertex $v\in\cV$ are then given by
\be\label{eq:kirchhoff}
\begin{cases} f\ \text{is continuous at}\ v,\\[1mm] 
\sum\limits_{e\in \cE_v}\partial_e f(v) =0. \end{cases}
\ee 
Imposing these boundary conditions on the maximal domain yields the \emph{maximal Kirchhoff Laplacian}
\be\label{eq:H}
\begin{split}
	\bH  =  -\Delta\upharpoonright {\dom(\bH)},\quad 
	 \dom(\bH ) = \{f\in H^2(\cG\setminus\cV)|\, f\ \text{satisfies}\ \eqref{eq:kirchhoff},\ v\in\cV\}.
\end{split}
\ee
Restricting further to compactly supported functions we end up with the pre-minimal operator
\be\label{eq:H00}
	\bH_{0}^0  =  -\Delta\upharpoonright {\dom(\bH_{0}^0)},\qquad 
	 \dom(\bH_{0}^0)  = \dom(\bH ) \cap L^2_c(\cG).
\ee
We call its closure $\bH_0 := \overline{ \bH_{0}^{0}}$ in $L^2(\cG)$ \emph{the minimal Kirchhoff Laplacian}. 
Integrating by parts one obtains
\be \label{eq:integrationbp}
	\langle\bH_0^0 f, f\rangle_{L^2(\cG)} = \int_\cG |\nabla f(x)|^2 \; \rD x =:\gt[f] , \qquad f \in \dom(\bH_0^0),
\ee
and hence both $\bH_0^0$ and $\bH_0$ are non-negative symmetric operators. It is known that
\be\label{eq:H0*=H}
		\bH_0^\ast = \bH.
\ee
The equality $\bH_0 = \bH$ holds if and only if $\bH_0$ is self-adjoint (or, equivalently, $\bH_0^0$ is essentially self-adjoint). 
To the best of our knowledge, the strongest sufficient condition which guaranties self-adjointness is provided by the next result.

\begin{theorem}[\cite{ekmn}]\label{th:ekmn}
Let $\varrho_m$ be the \emph{star path metric} on $\cV$, 
\be\label{def:rho_m}
\varrho_m(u,v) := \inf_{\substack{\cP=(v_0,\dots,v_n)\\ u=v_0,\ v=v_n}}\sum_{v_k\in\cP} m(v_k),
\ee
where $m\colon \cV\to (0,\infty)$ is the \emph{star weight} 
\be\label{def:m}
		m(v) := \sum_{e \in \cE_v}|e| = \vol(\cE_v).
\ee
If $(\cV,\varrho_m)$ is complete as a metric space, then $\bH_0^0$ is essentially self-adjoint.
\end{theorem}

The degree of non-self-adjointness of $\bH_0$ is determined by its \emph{deficiency indices} $\Nr_\pm(\bH_0) = \dim\cN_{\pm\I}(\bH_0)$, where  
\begin{align}\label{def:n_pm}
\cN_z(\bH_0) := \ker(\bH_0^\ast - z)= \ker(\bH - z),\quad z\in\C,
\end{align}
are called the \emph{deficiency subspaces} of $\bH_0$. Notice that $\Nr_+(\bH_0) = \Nr_-(\bH_0)$ since $\bH_0$ is non-negative. It is self-adjoint exactly when $\Nr_\pm(\bH_0) = 0$. Otherwise, according to the von Neumann formulas \cite[Theorem 13.10]{schm}, the self-adjoint extensions of $\bH_0$ can be parameterized by the unitary group ${\rm U}(\Nr)$ with $\Nr=\Nr_\pm(\bH_0)$. 

There is a standard procedure to construct at least one self-adjoint extension of $\bH_0$, the so-called {\em Friedrichs extension}, let us denote it by $\bH_D$. Namely, $\bH_D$ is defined as the operator associated with the closure in $L^2(\cG)$ of the quadratic form \eqref{eq:integrationbp}. Clearly, the domain of the closure coincides with $H^1_0(\cG)$ and hence $\bH_D$ is given as the restriction of $\bH$ to the domain $\dom(\bH_D):= \dom(\bH)\cap H^1_0(\cG)$ (see, e.g., \cite[Theorem 10.17]{schm}). Taking into account \eqref{eq:h10}, $\bH_D$ is often called the {\em Dirichlet Laplacian} (which also explains the subscript). On the other hand, the form $\gt$ is well defined on $H^1(\cG)$ and, moreover,
\[
\gt_N[f] := \gt[f],\qquad f\in \dom(\gt_N) = H^1(\cG)
\]
is closed (since $H^1(\cG)$ is a Hilbert space). The self-adjoint operator $\bH_N$ associated with this form is usually called the {\em Neumann extension} of $\bH_0$ (not to confuse it with Krein--von Neumann extension!) or {\em Neumann Laplacian}.

\begin{remark}
Following the analogy with the Dirichlet Laplacian, it might be tempting to conjecture that the domain of the Neumann Laplacian is given by $\dom(\bH)\cap H^1(\cG)$. However, the operator defined on this domain has a different name --- the Gaffney Laplacian -- and it is not even symmetric in general. The main focus of the following two sections will be on the study of this operator.
\end{remark}

%%%%%%%%%%%%%%%%
\section{The Gaffney Laplacian}\label{sec:Gaffney}
%%%%%%%%%%%%%%%%

Let us fix an orientation on $\cG$. 
In the Hilbert space $L^2(\cG)$, we can associate (at least) two gradient operators with $\nabla$ defined by \eqref{eq:nabla}. Namely, set 
\begin{align}\label{eq:nablaDN}
\nabla_D & := \nabla\upharpoonright\dom(\nabla_D), & \nabla_N & := \nabla\upharpoonright\dom(\nabla_N),
\end{align}
where
\begin{align}\label{eq:dom_nabla}
\dom(\nabla_D) & = H^1_0(\cG), & \dom(\nabla_N) & = H^1(\cG).
\end{align}
The importance of $\nabla_D$ and $\nabla_N$ stems from the following fact.

\begin{lemma}\label{lem:H=nabla}
Let $\bH_D$ and $\bH_N$ be the Friedrichs and Neumann extensions of $\bH_0$, respectively. Then
\begin{align}\label{eq:FNviaNabla}
\bH_D &= \nabla_D^\ast \nabla_D, & \bH_N & = \nabla_N^\ast \nabla_N,
\end{align}
where $\ast$ denotes the adjoint operator\footnote{Here and below the product $AB$ of two unbounded operators $A$, $B$ is understood as their composition: $(AB)(f):= A(Bf)$ for all $f\in \dom(AB):= \{f\in \dom(B)|\, Bf\in \dom(A)\}$.}. 
\end{lemma}

\begin{proof}
Since $H^1_0(\cG)$ and $H^1(\cG)$ are Hilbert spaces, both $\nabla_D$ and $\nabla_N$ are closed operators in $L^2(\cG)$ and hence, by von Neumann's theorem \cite[Chapter~V.3.7]{kato}, $\nabla_D^\ast \nabla_D$ and $\nabla_N^\ast \nabla_N$ are self-adjoint non-negative operators in $L^2(\cG)$. The quadratic forms associated with $\nabla_D^\ast \nabla_D$ and $\nabla_N^\ast \nabla_N$ coincide with, respectively, the quadratic forms of $\bH_D$ and $\bH_N$ and the claim now follows from the representation theorem (see, e.g., \cite[Chapter VI.2.1]{kato}).
\end{proof}

\begin{remark}
Clearly, $\nabla$ and hence both $\nabla_D$ and $\nabla_N$ do depend on the choice of an orientation on $\cG$. However, it is straightforward to see that the second order operators $\bH_D$ and $\bH_N$ do not depend on it.
\end{remark}

Now we are in position to introduce the main object. In the Hilbert space $L^2(\cG)$, define the following operators
\begin{align}\label{eq:GaffneyDef}
\bH_{G,\min} &= \nabla_N^\ast \nabla_D, & \bH_G & = \nabla_D^\ast \nabla_N.
\end{align}
They act edgewise as the negative second derivative and their domains are
\begin{align*} 
\dom(\bH_{G,\min}) &= \{f\in H^1_0(\cG)|\, \nabla f\in \dom(\nabla_N^\ast)\}, \\ 
\dom(\bH_{G}) &= \{f\in H^1(\cG)|\, \nabla f\in \dom(\nabla_D^\ast)\}.
\end{align*}
The operator $\bH_G$ is called the {\em Gaffney Laplacian}. We shall refer to $\bH_{G,\min}$ as the {\em minimal Gaffney Laplacian}. 

\begin{remark}\label{rem:GafDef}
Notice that the above definition is not precisely the original definition of M.\ P.\ Gaffney \cite{gaf} (roughly speaking $H^1$ was replaced in \cite{gaf,gaf55} by $C^1\cap H^1$, see also \cite{ma99,ma05}). The obvious drawback is that the corresponding maximal Laplacian is always non-closed.
 Let us also stress that we are unaware of $\bH_{G,\min}$ in the manifold context and this natural, in our opinion, object seems to be new.
\end{remark}

\begin{remark}\label{rem:Hodge}
One can introduce $0$-forms and $1$-forms on $\cG$ and, upon assigning an orientation, both can be further identified with functions (see, e.g., \cite{p09}).\footnote{Due to the local 1d nature of metric graphs, the space of $2$-forms on $\cG$ is trivial.} 
From this perspective the operator 
\[
	\vec{\Delta} = \nabla_N \nabla_D^\ast
\]
is a metric graph analogue of the Hodge Laplacian on $1$-forms (see \cite[\S 5.1]{bake19}, \cite{go91}, \cite{p09}). 
Indeed (see, e.g., \cite{gil}), the Hodge Laplacian on smooth $k$-forms is given by  
\[
	\Delta_k = \delta^{k+1} \rD^k + \rD^{k-1} \delta^k,
\]
where $\rD^{k}$ is the exterior derivative (mapping $k$-forms to $(k+1)$-forms) and the co-differential $\delta^{k+1}$ is its formal adjoint (mapping $(k+1)$-forms to $k$-forms). Working in the $L^2$-framework and replacing smooth by $H^1$ for metric graphs, one can identify $\rD^1 = \nabla_N$ and $\delta^0 = \nabla_D^\ast$. In particular, the Gaffney Laplacian \eqref{eq:GaffneyDef} 
can be viewed as the Hodge Laplacian on $0$-forms. Let us also stress that due to the supersymmetry, 
the properties of $\bH_G$ and $\vec{\Delta}$ are closely connected.
\end{remark}

\begin{lemma}\label{lem:restr-ext}
Both operators $\bH_G$ and $\bH_{G,\min}$ are restrictions of the maximal Kirchhoff Laplacian $\bH$/extensions of the minimal Kirchhoff Laplacian $\bH_0$,
\begin{align}\label{eq:inclusGKH}
\bH_0 \subseteq \bH_{G,\min} \subseteq  \bH_G \subseteq \bH.
\end{align}
\end{lemma}

\begin{proof}
It is straightforward to verify both claims, however, we would like to show that  
\be\label{eq:domnabla*}
\dom(\nabla_D^\ast) = \Big\{f\in H^1(\cG\setminus\cV)\Big|\, 
             \sum_{e\in \cE_v} \vec{f_e}(v) =0\ \text{for all}\ v\in\cV \Big\},
\ee
which then makes the inclusions in \eqref{eq:inclusGKH} obvious. Here we employ the following notation
\[
\vec{f_e}(v) = \begin{cases}\ \ f_e(v), & v \ \text{is terminal}, \\ -f_e(v), & v \ \text{is initial}. \end{cases}
\]

If $f$ belongs to the RHS in~\eqref{eq:domnabla*}, then an integration by parts gives
\be\label{eq:ell_f}
\ell_f(g) := \langle \nabla_D g, f\rangle_{L^2} = -\langle  g, \nabla f\rangle_{L^2}
\ee
for all $g\in H^1(\cG)\cap L^2_c(\cG)$. Clearly, $\ell_f$ extends to a bounded linear functional on $L^2$, which implies that $f\in \dom(\nabla_D^\ast)$ and hence this proves the inclusion ``$\supseteq$". 

Suppose now that $f\in \dom(\nabla_D^\ast)$.  Fixing an edge $e \in \cE$ and taking a test function $g\in H^1_0(\cG)$ such that $g$ equals zero everywhere except $e$, we immediately conclude that $f$ belongs to $H^1$ on $e$. 
Next pick a vertex $v \in \cV$.  Choose $g\in H^1_0(\cG)$ such that $g \equiv 0$ on $\cE\setminus\cE_v$. Moreover, for every $e\in\cE_v$ we assume that $g(x_e) = 1$ if $x_e \in e$ and $|x_e - v| < |e|/4$ and $g(x_e) = 0$ if $|x_e - v| > |e|/2$. Thus we get
	\begin{align*}
		0  = \langle f, \nabla_D g \rangle - \langle \nabla_D^\ast f, g\rangle \
		&= \sum_{e \in \cE_v} \int_e f(\nabla g)^\ast + \nabla f\, g^\ast dx_e\\
		&= \sum_{e \in \cE_v} \vec{f_e}(v) g_e(v)^\ast 
		 = \sum_{e \in \cE_v} \vec{f_e}(v).
	\end{align*}
This also implies that we can perform the integration by parts in \eqref{eq:ell_f} for every  $g\in H^1(\cG)\cap L^2_c(\cG)$. Since $\ell_f$ extends to a bounded linear functional on $L^2(\cG)$, we conclude that $\nabla f\in L^2(\cG)$, which completes the proof.
\end{proof}

Clearly, all four operators in \eqref{eq:inclusGKH} coincide exactly when $\bH_0$ is self-adjoint (and hence all four operators are self-adjoint). 
 Moreover, by the very definition we have 
\begin{align}\label{eq:FNvsGaff}
\bH_{G,\min} &\subseteq \bH_D \subseteq \bH_G, & \bH_{G,\min} &\subseteq \bH_N \subseteq \bH_G.
\end{align}
In particular, $\bH_{G,\min}$ is symmetric, however, $\bH_G$ may not be symmetric (and hence self-adjoint).
The next result provides several self-adjointness criteria for $\bH_G$ including a transparent geometric characterization.

\begin{lemma}\label{cor:Markov}
The following statements are equivalent:
\begin{itemize}
\item[(i)] The Gaffney Laplacian $\bH_G$ is self-adjoint,
\item[(ii)] $\nabla_N = \nabla_D$,
\item[(iii)] $H^1_0(\cG) = H^1(\cG)$,
\item[(iv)] $\bH_0$ has a unique Markovian extension,
\item[(v)] $\cG$ has no finite volume ends, $\gC_0(\cG) = \emptyset$.
\end{itemize}
\end{lemma}

\begin{proof}
The equivalence $(iii)\Leftrightarrow (iv)$ is well known; $(iii)\Leftrightarrow (v)$ was established in \cite[Corollary 3.12]{kmn19}. The remaining equivalences follow upon noting that $\nabla_D=\nabla_N$ if and only if $H^1_0(\cG) = H^1(\cG)$.
\end{proof}

Recall that an extension $\wt\bH$ of the minimal Kirchhoff Laplacian $\bH_0$ is called \emph{Markovian} if $\wt\bH$ is a non-negative self-adjoint extension and the corresponding quadratic form is a \emph{Dirichlet form} (for further details we refer to \cite[Chapter 1]{fuk10}).  Hence the associated semigroup $\E^{-t\wt\bH}$, $t>0$ as well as resolvents $(\wt\bH + \lambda)^{-1}$, $\lambda>0$ are Markovian: i.e., are both \emph{positivity preserving} (map non-negative functions to non-negative functions) and \emph{$L^\infty$-contractive} (map the unit ball of $L^\infty(\mathcal G)$ into itself). Notice that both the Dirichlet and Neumann Laplacians are Markovian extensions.
A self-adjoint extension $\wt{\bH}$ of $\bH_0$ is called a {\em finite energy extension} if its domain is contained in $H^1(\cG)$. In particular, every Markovian extension is a finite energy extension (for further details we refer to \cite[\S 5]{kmn19}). 

The importance of the  Gaffney Laplacian  in the study of Markovian and, more generally, finite energy extensions of $\bH_0$ stems from the following fact.

\begin{lemma}
The domain of the Gaffney Laplacian is given by
\begin{align}\label{eq:domHG}
\dom(\bH_G) = \dom(\bH)\cap H^1(\cG).
\end{align}
In particular, $\wt{\bH}$ is a Markovian/finite energy self-adjoint extension of $\bH_0$ if and only if $\wt{\bH}$ is a Markovian/self-adjoint restriction of $\bH_G$.
\end{lemma}

\begin{proof}
The inclusion $\dom(\bH_G) \subseteq \dom(\bH)\cap H^1(\cG)$ follows from the definition of $\bH_G$. The converse inclusion is immediate from~\eqref{eq:domnabla*}.
\end{proof}

\begin{remark}
The equivalences $(i)\Leftrightarrow(ii)\Leftrightarrow(iii)\Leftrightarrow(iv)$ are well known in the manifold context, where the most subtle part is to show that the Dirichlet and Neumann Laplacians are extremal extensions in the set of all Markovian extensions (cf. \cite[Theorem 1.7 and \S 3.5]{gm}). Let us also stress that the proper analog of the equality $\nabla_N = \nabla_D$ in the manifold context should read as ``the gradient is a formal skew-adjoint of the divergence operator" (see Remark \ref{rem:Hodge}) and this property was called {\em negligible boundary} in \cite{gaf,gaf55}. Moreover, the central result of \cite{gaf} is that negligible boundary is sufficient for the (essential) self-adjointness of the Gaffney Laplacian (notice that the quadratic forms approach was not available at that time).
\end{remark}

As in the case of the maximal and minimal Kirchhoff Laplacians, there is a close connection between the Gaffney Laplacians.
\begin{lemma}\label{lem:minGaffClosed}
 The minimal Gaffney Laplacian is closed in $L^2(\cG)$ and 
\be\label{eq:HGmin=ast}
\bH_{G,\min} = \bH_G^\ast.
\ee
\end{lemma}

\begin{proof}
By the very definition of $\bH_G$, we get 
\[
\nabla_N^\ast \nabla_D \subseteq \bH_G^\ast 
\]
and hence $\bH_{G,\min}\subseteq \bH_G^\ast$. To prove the converse inclusion, observe that $\bH_G^\ast\subseteq \bH_D$ and $\bH_G^\ast\subseteq \bH_N$. Taking into account Lemma~\ref{lem:H=nabla}, we thus get
\begin{align*}
\dom(\bH_{G}^\ast) & \subseteq \dom(\bH_D)\cap\dom(\bH_N) \\
& = \{f\in H^1_0(\cG)|\, \nabla_D f\in \dom(\nabla_D^\ast)\}\cap\{f\in H^1(\cG)|\, \nabla_N f\in \dom(\nabla_N^\ast)\}\\
& = \{f\in H^1_0(\cG)|\, \nabla_D f\in \dom(\nabla_N^\ast)\} \\
&= \dom(\bH_{G,\min}),
\end{align*}
which implies~\eqref{eq:HGmin=ast}. Since the adjoint is always closed, the first claim follows.
\end{proof}
 
\begin{remark}
The fact that $\bH_{G,\min}$ is a closed operator can be seen by noting that its domain is the intersection of domains of the Friedrichs and Neumann extensions, 
 \be\label{eq:Min=FcapN}
 \dom(\bH_{G,\min}) = \dom(\bH_D)\cap \dom(\bH_N).
 \ee
 Indeed, one simply needs to compare the definition of $\bH_{G,\min}$ with~\eqref{eq:FNviaNabla} and take into account that $\nabla_N^\ast\subseteq \nabla_D^\ast$. Since both $\bH_D$ and $\bH_N$ are closed, one concludes that so is $\bH_{G,\min}$.
\end{remark}
 
 Our first main result is the following connection between deficiency indices of $\bH_{G,\min}$ and graph ends.
 
 \begin{theorem}\label{th:n-pm=C0}
The deficiency indices of the minimal Gaffney Laplacian $\bH_{G,\min}$ coincide with the number of finite volume ends of $\cG$,
\be\label{eq:n_pm=C0}
\Nr_\pm(\bH_{G,\min}) = \#\gC_0(\cG).
\ee
 \end{theorem}
 
 \begin{proof}
 If $\#\gC_0(\cG) <\infty$, then using the results of \cite[\S 6]{kmn19} one can easily see that $\dim(\dom(\bH_N)/\dom(\bH_{G,\min})) = \#\gC_0(\cG)$ (indeed, combine Prop.~6.6(i) and Corollary 6.7 with Theorem 3.11 from \cite{kmn19}). This immediately implies~\eqref{eq:n_pm=C0} since $\bH_N$ is a self-adjoint extension of $\bH_{G,\min}$ (cf.~\cite[Theorem~13.10]{schm}).

It remains to consider the case $\#\gC_0(\cG) = \infty$. By~\eqref{eq:domHG} and \cite[Lemma 4.7]{kmn19}, 
\be\label{eq:defsubHG}
\dim(\ker(\bH_G - \lambda)) = \dim(\ker(\bH - \lambda) \cap H^1(\cG)) =  \# \gC_0(\cG)
\ee
for all negative real $\lambda$. 
  It remains to notice that $\bH_G\subseteq \bH_{G,\min}^\ast$ and hence the claim follows from the positivity of $\bH_{G,\min}$.
 \end{proof}
 
In contrast to the minimal Gaffney Laplacian, $\bH_G$ is not automatically closed, that is, $\bH_{G,\min}^\ast = \overline{\bH_G}$, although it is not necessarily true that
\be\label{eq:HG=ast}
\bH_{G,\min}^\ast =\bH_G.
\ee
Our second main result provides necessary and sufficient conditions for the Gaffney Laplacian to be closed.

\begin{theorem}\label{th:main}
Let $\cG$ be a metric graph satisfying Hypothesis \ref{hyp:locfin} and let $\bH_G$ be the corresponding Gaffney Laplacian.
\begin{itemize}
\item[(i)]  If $\#\gC_0(\cG)<\infty$, then $\bH_G$ is closed and~\eqref{eq:HG=ast} holds true.
\item[(ii)] If $\cG$ contains a non-free finite volume end, then $n_\pm(\bH_{G,\min}) = \infty$ and $\bH_G$ is not closed.
\end{itemize}
\end{theorem}

\begin{proof}
(i) It suffices to employ the decomposition
\be \label{eq:domHGviaF}
	\dom(\bH_{G,\min}^\ast) = \dom(\bH_D) \dotplus \ker(\bH_{G,\min}^\ast - z) = \dom(\bH_D) \dotplus \cN_z(\bH_{G,\min}),
\ee
which holds for every $z$ in the resolvent set of $\bH_D$ (see, e.g., \cite[Prop.~14.11]{schm}).
By Theorem~\ref{th:n-pm=C0}, $\dim (\ker(\bH_{G,\min}^\ast - z)) = \#\gC_0(\cG)<\infty$. Combining~\eqref{eq:defsubHG} with $\dom(\bH_G)\subseteq \dom(\bH_{G,\min}^\ast)$, we conclude that $\dom(\bH_G) = \dom(\bH_{G,\min}^\ast)$. \\

(ii) 
Since the mapping $\nabla \colon H^1(\cG) \to L^2(\cG)$ is bounded, \eqref{eq:HG=ast} holds if and only if there exists a positive constant $C > 0$ such that
\begin{align} \label{eq:sobolev}
	\| \nabla f\|_{L^2(\cG)}^2 \le C (\|  f\|_{L^2(\cG)}^2 + \| \bH f\|_{L^2(\cG)}^2),  
\end{align}  
for all $f \in \dom(\bH_G) = \dom(\bH)\cap H^1(\cG)$. 

Suppose $\cG$ has a non-free end of finite volume $\gamma$ and fix a sequence $\cU = (U_n)$ of open sets representing $\gamma$. We can choose a sequence $(\cG_n)$ of connected subgraphs of $\cG$ such that $\# \partial  \cG_n < \infty$, almost all of the open sets $U_k$ are contained in $\cG_n$ and $\cG_n \subset U_n$ for all $n$. The latter two properties imply that $\#\gC_0(\cG_n) = \infty$ for every $n\ge 0$ and $\bigcap_{n\ge 0} \cG_n = \emptyset$. In particular, $\vol(\cG_n) \to 0$ as $n\to \infty$. 
 
Fix $\lambda<0$ and denote by $\bH_n$ the Gaffney Laplacian on the subgraph $\cG_n$, $n\ge 0$. Taking into account \eqref{eq:defsubHG}, there exists a real-valued function $f_n \in \ker(\bH_n - \lambda)$ with $f_n(v) = 0$ for all $v \in \partial \cG_n$ and such that $\| f_n \|_\infty = 1$ Moreover, extending $f_n$ by zero on $\cG \setminus \cG_n$ gives a function (also denoted by $f_n$) belonging to the domain of the Gaffney Laplacian $\bH_G$ on $\cG$. 

Assuming that $\bH_G$ is closed,  \eqref{eq:sobolev} would imply that 
\[ 
\| \nabla f_n \|_{L^2(\cG)} \lesssim \| f_n\|_{L^2(\cG)}^2 + \| \bH f_n\|_{L^2(\cG)}^2 = (1 +\lambda^2) \| f_n\|_{L^2(\cG_n)}^2 
\le (1 +\lambda^2)   \vol(\cG_n)
\]
for all $n\ge 0$. Next, for each $n$ there is  $x_n \in \cG_n$ with $|f_n(x_n)| \ge 1/2$. Choosing $y_n \in \partial \cG_n$ and a path $\cP_n$ in $\cG_n$ connecting $x_n$ and $y_n$, we get
\[
	\frac{1}{2} =  |f_n(x_n) - f_n(y_n)|\le \int_{\cP_n} |\nabla f_n (x)| dx \le \vol (\cP_n) \| \nabla f_n \|_{L^2(\cG)} \lesssim \vol (\cG_n)^2
\]
for all $n\ge 0$. However, the right-hand side tends to zero when $n\to \infty$. This contradiction completes the proof.
\end{proof}

Let us now present two particular cases of interest when Theorem \ref{th:main} provides a necessary and sufficient condition for $\bH_G$ to be closed.
 
 \begin{corollary}\label{cor:main} 
 Suppose $\cG$ has finite total volume. The following are equivalent:
 \begin{itemize}
 \item[(i)] The Gaffney Laplacian $\bH_G$ is closed,
 \item[(ii)]  \eqref{eq:HG=ast} holds true,
 \item[(iii)] $\cG$ has finitely many ends, $\#\gC(\cG)<\infty$.
 \end{itemize}
 \end{corollary}

\begin{proof}
We only need to notice that $\gC(\cG) = \gC_0(\cG)$ in the case when $\vol(\cG)<\infty$. By Halin's theorem  \cite{hal}, a locally finite graph $\cG$ has at least one end which is not free if $\gC(\cG)=\infty$. Thus, it remains to apply Theorem~\ref{th:main}.
\end{proof}

Theorem \ref{th:main} also gives rise to a criterion in the case of Cayley graphs.

 \begin{corollary}\label{cor:group} 
 Suppose $\cG_d$ is a Cayley graph of a countable finitely generated group $\mG$. Then $\bH_G$ is not closed if and only if $\#\gC(\cG) = \infty$ and $\cG$ has at least one finite volume end.
 \end{corollary}

\begin{proof}
If there are infinitely many ends, then the end space is known to be homeomorphic to the Cantor set (see, e.g., \cite[Addendum~13.5.8]{geog}), and hence there are no free graph ends. Theorem \ref{th:main} completes the proof.
\end{proof}

\begin{remark}
By the Freudenthal--Hopf theorem, a Cayley graph of a countable finitely generated group has 1, 2 or infinitely many ends. Moreover, the number of ends is independent of the choice of the finite generating set. 
By Hopf's theorem, $\cG_d$ has exactly two ends if and only if $\mG$ is virtually infinite cyclic. The classification of finitely generated groups with infinitely many ends is due to J.\ R.\ Stallings (see, e.g., \cite[Chapter 13]{geog}). In particular, if $\mG$ is amenable, then it has finitely many ends (actually, either 1 or 2) and hence the Gaffney Laplacian is always closed for Cayley graphs of amenable groups.
\end{remark}

\begin{remark}
The above considerations shed more light on the results obtained in \cite{kmn19}. 
\begin{itemize}
\item[(i)] First of all, combining \eqref{eq:inclusGKH} with \eqref{eq:domHG} and Theorem \ref{th:main}(i) we obtain one of the main results in \cite{kmn19}, Theorem 4.1, on the deficiency indices of the Kirchhoff Laplacian: 
\be\label{eq:npm>ends}
		\Nr_\pm (\bH_0) \ge \# \gC_0(\cG),
\ee
{\em with equality if and only if either $\# \gC_0(\cG) = \infty$ or $\dom(\bH) \subset H^1(\cG)$}.
\item[(ii)] It is straightforward to see that in the case when $\cG$ has finitely many finite volume ends, $\#\gC_0(\cG)<\infty$, the triplet $\Pi = \{\C^{\#\gC_0(\cG)}, \Gamma_0,\Gamma_1\}$, where the mappings $\Gamma_0,\Gamma_1\colon \dom(\bH)\cap H^1(\cG)\to \C^{\#\gC_0(\cG)}$ are defined by
\begin{align}\label{def:Gamma}
\Gamma_0\colon & f\mapsto  \big(f(\gamma)\big)_{\gamma \in \gC_0 (\cG)}, & 
\Gamma_1\colon & f\mapsto \big(\partial_n f (\gamma)\big)_{\gamma \in \gC_0 (\cG)},
\end{align}
(see Prop.~6.6 and Lemma 6.9 in \cite[\S 6]{kmn19}) is a {\em boundary triplet}\footnote{For definitions and basic properties we refer to, e.g., \cite[Chapter 14]{schm} or \cite[Appendix A]{ekmn}.} for the Gaffney Laplacian $\bH_G$. This also implies the description of Markovian and finite energy extensions of $\bH_0$ obtained in \cite[Theorem 6.11]{kmn19}.
\end{itemize}
\end{remark}

%%%%%%%%%%%%%%%%%%%%%%%%
%%%%%%%%%%%%%%%%%%%%%%%%
\section{Examples} \label{sec:examples}
%%%%%%%%%%%%%%%%%%%%%%%%
%%%%%%%%%%%%%%%%%%%%%%%%
The case not covered by Theorem \ref{th:main} is when $\cG$ has infinitely many finite volume free ends, however, all non-free ends have infinite volume. Moreover, there is one additional problem: it is not clear what is the closure of $\bH_G$ if it is not closed. 
We begin with the following result.

\begin{proposition} \label{prop:graphseq}
Let  $\cG$ contain a sequence of connected subgraphs $(\cG_n)$ such that 
\begin{itemize}
\item[(a)] $\lim_{n\to \infty} \vol(\cG_n) = 0$, and 
\item[(b)] $\# \partial \cG_n < \# \gC(\cG_n)$ for all $n\ge 0$.
\end{itemize}
 Then $\bH_G$ is not closed.
\end{proposition} 

\begin{proof}
It is easy to see that properties (a) and (b) are exactly the ones used in the proof of Theorem \ref{th:main}(ii) (in the case of a non-free finite volume end $\# \partial \cG_n < \infty$ and $\# \gC(\cG_n) = \# \gC_0(\cG_n) = \infty$ for all $n\ge 0$) and hence the proof of Proposition \ref{prop:graphseq} is literally the same and we leave it to the reader.
\end{proof}

\begin{remark}\label{rem:graphseq}
In fact, one can replace (a) in Proposition \ref{prop:graphseq} by %the weaker assumption:
\begin{itemize}
\item[(a')]  $\sup_n \vol(\cG_n) < \infty$ and $\lim_{n\to \infty} \diam(\cG_n) = 0$,
\end{itemize}
where $\diam(\cG_n) = \sup_{x,y\in \cG_n}\varrho(x,y)$ is the diameter of $\cG_n$.
\end{remark}

Proposition \ref{prop:graphseq} enables us to construct graphs without finite volume non-free ends, however with the corresponding Gaffney Laplacian $\bH_G$ not being closed. 

\begin{example} 
Take a path graph $\cP_0 = (\Z_{\ge 0}, |\cdot|)$ equipped with some positive lengths and attach to each vertex $v_n \in \Z_{\ge 0}$ an infinite rooted graph $\cG_n$. If each $\cG_n$ has at least two ends (e.g., each $\cG_n$ consists of two rooted antitrees joined at the root vertices) and $\liminf_n \vol (\cG_n) = 0$, then $\bH_G$ is not closed (see Figure \ref{fig:line}).
\end{example}

\begin{figure}[h!] 
\begin{center}
	\begin{tikzpicture}
	
		%\draw (-4, 0) -- (-3, 0) [dashed];	
		\draw (5, 0) -- (6, 0) [dashed];	
		\draw (4, 0) -- (5, 0);
		%%% LINE
		\foreach \x in {0, 2}{ 
		
		\draw  (\x , 0) -- (\x +2 , 0);}

		\foreach \x in {0, 2, 4}{ \filldraw (\x, 0)  circle (0.7 pt) node [below left] {\footnotesize  \;}  ;
		
		}
		
		\foreach \x in {0, 2}{ 
		
		%Up graph
		\draw  (\x , 0) -- (\x , {0.3/(0.25 *\x + 1)});
		\fill [pattern=north east lines, pattern color=gray]  (\x , {0.3/(0.25 *\x + 1)}) -- ({\x + 0.5 /(0.25 *\x + 1)}, {0.6/(0.25 *\x + 1)}) -- ({\x + 0.5 /(0.25 *\x + 1)}, {0.9/(0.25 *\x + 1)}) -- ({\x - 0.5 /(0.25 *\x + 1)}, {0.9/(0.25 *\x + 1)}) -- ({\x - 0.5 /(0.25 *\x + 1)}, {0.6/(0.25 *\x + 1)})	 --  (\x , {0.3/(0.25 *\x + 1)});
		
		\filldraw (\x , {0.3/(0.25 *\x + 1)}) circle (0.7 pt) node [below left] {\footnotesize  \;}  ;

		%Down graph
		\draw  (\x , 0) -- (\x , {-0.3/(0.25 *\x + 1)});
		\fill [pattern=north east lines, pattern color=gray]  (\x , {-0.3/(0.25 *\x + 1)}) -- ({\x + 0.5 /(0.25 *\x + 1)}, {-0.6/(0.25 *\x + 1)}) -- ({\x + 0.5 /(0.25 *\x + 1)}, {-0.9/(0.25 *\x + 1)}) -- ({\x - 0.5 /(0.25 *\x + 1)}, {-0.9/(0.25 *\x + 1)}) -- ({\x - 0.5 /(0.25 *\x + 1)}, {-0.6/(0.25 *\x + 1)})	 --  (\x , {-0.3/(0.25 *\x + 1)});
		
		\filldraw (\x , {-0.3/(0.25 *\x + 1)}) circle (0.7 pt) node [below left] {\footnotesize  \;}  ;
}
		
		\foreach \x in {4}{ 
		\draw[dotted, thin, gray] (\x , 0) -- (\x , {0.3/(0.25 *\x + 1)});
		\fill [pattern=north east lines, pattern color=gray]  (\x , {0.3/(0.25 *\x + 1)}) -- ({\x + 0.5 /(0.25 *\x + 1)}, {0.6/(0.25 *\x + 1)}) -- ({\x + 0.5 /(0.25 *\x + 1)}, {0.9/(0.25 *\x + 1)}) -- ({\x - 0.5 /(0.25 *\x + 1)}, {0.9/(0.25 *\x + 1)}) -- ({\x - 0.5 /(0.25 *\x + 1)}, {0.6/(0.25 *\x + 1)})	 --  (\x , {0.3/(0.25 *\x + 1)});
		
		\draw[dotted, thin, gray] (\x , 0) -- (\x , {-0.3/(0.25 *\x + 1)});
		\fill [pattern=north east lines, pattern color=gray]  (\x , {-0.3/(0.25 *\x + 1)}) -- ({\x + 0.5 /(0.25 *\x + 1)}, {-0.6/(0.25 *\x + 1)}) -- ({\x + 0.5 /(0.25 *\x + 1)}, {-0.9/(0.25 *\x + 1)}) -- ({\x - 0.5 /(0.25 *\x + 1)}, {-0.9/(0.25 *\x + 1)}) --({\x - 0.5 /(0.25 *\x + 1)}, {-0.6/(0.25 *\x + 1)})	 --  (\x , {-0.3/(0.25 *\x + 1)});
			%\filldraw (\x , {0.5/(0.3 *\x + 1)}) circle (0.9 pt) node [below left] {\footnotesize  \;}  ;
		}
		
		\foreach \x in {0, 1, 2}{ 
		
		\node at ({2*\x}, 0.9) [above] {\footnotesize $\cG_{\x}$} ;
		}
			\end{tikzpicture}
\end{center}
	\caption{$\Z_{\ge 0}$ with attached graphs $\cG_n$.}
	\label{fig:line}
\end{figure}
%%%%%%%%%%%%%%%%%%%%%%%%%%%%%

\begin{remark}
If one of the conditions (a), (a'), or (b) fails to hold, then the corresponding Gaffney Laplacian may or may not be closed. Indeed, consider the graph depicted on Figure \ref{fig:line}. Assuming that each $\cG_n$ has finitely many graph ends and finite total volume, the Gaffney Laplacian $\bH_n$ on $\cG_n$ is closed, which is further equivalent to the validity of the Sobolev-type inequality \eqref{eq:sobolev} with some constant $C_n>0$, $n\ge 0$. Also one has similar inequalities on every edge of a path graph, however, the corresponding constants do depend on the edges lengths (see, e.g., \cite[Chapter IV.2]{kato}). When we ``glue" the graphs $(\cG_n)$ and edges of a path graph, the space $H^1(\cG)$ is a subspace of the direct sum of $H^1$ spaces and hence if all the constants admit a uniform upper bound, then \eqref{eq:sobolev} would trivially hold true on $\cG$ (e.g., take all $\cG_n$ being identical and assume that the path graph is equilateral). 

Obtaining a complete answer in the case of a metric graph depicted on Figure \ref{fig:line} seems to be an interesting and nontrivial problem.
\end{remark}

In conclusion we would like to show that for a large class of metric graphs the closure of the Gaffney Laplacian may coincide with the maximal Kirchhoff Laplacian (which is also equivalent to the fact that $\bH_0 = \bH_{G,\min}$).

\begin{example}[Radially symmetric trees]
Let $\cG= \cT$ be a \emph{radially symmetric metric tree}: that is, a tree $\cT$ with a root $o$ such that for each $n \ge 0$, all vertices in the combinatorial sphere $S_n$ have the same number of descendants $b_n \in \Z_{\ge 2}$ and all edges between $S_n$ and $S_{n+1}$ have the same length $\ell_n\in (0,\infty)$. Clearly, a radially symmetric tree $\cT$ is uniquely determined by the sequences $(b_n)$ and $(\ell_n)$. The assumptions imply that $\cT$ has uncountably many ends. Define
\begin{align*}
\mu_n & = \prod_{k=0}^n b_k, & t_n  & = \sum_{k=0}^{n-1} \ell_k,
\end{align*}
for all $n\ge 0$. Notice that $(\cT,\varrho)$ is complete exactly when $\cL:= \lim_{n\to \infty} t_n = \infty$. 

\begin{lemma}\label{lem:HGtree}
Let $\cG=\cT$ be a radially symmetric tree. 
\begin{itemize}
\item[(i)] The corresponding Gaffney Laplacian $\bH_G$ is self-adjoint if
\begin{align*}
\vol(\cT) = \sum_{n\ge 0} \mu_n\ell_n = \infty.
\end{align*}
\item[(ii)]
If $\vol(\cT)<\infty$, then $\bH_G$ is not closed and its closure coincides with the maximal Kirchhoff Laplacian $\bH$, $\bH_G\neq \overline{\bH}_G = \bH$.
\end{itemize}
\end{lemma}

\begin{proof}
(i) It is well-known (see, e.g., \cite{car00, sol04}) that in the  radially symmetric case $\bH$ is self-adjoint if and only if $\vol(\cT) = + \infty$. In particular, in this case all four operators in \eqref{eq:inclusGKH} coincide and hence $\bH_G = \bH_G^\ast$ is closed. 

(ii) If  $\mathcal{T}$ has finite volume, then by Corollary \ref{cor:main} the Gaffney Laplacian $\bH_G$ is not closed. 
Since $\vol(\cT) < \infty$, the Friedrichs extensions $\bH_D$ has strictly positive spectrum (e.g. \cite[Corollary 3.5]{kn17}) and hence 
\[ 
	\dom(\bH) = \dom(\bH_D) \dotplus \ker(\bH).
\]
However, $\dom(\bH_D) \subseteq \dom(\bH_G)$ and it suffices to show that $\ker(\bH)\subseteq \dom( \overline{\bH}_G)$.
According to \cite{car00, ns01} (see also \cite[Section 7]{sol04} and \cite{bl19}), the Kirchhoff Laplacian on a radially symmetric tree $\bH$ is unitarily equivalent to 
\be \label{eq:decompKHL}
	\wt\bH = \rH_{\sym} \bigoplus_{n\ge 0} {\rH}_{n}\otimes \mathbb{I}_{\mu_n-\mu_{n-1}}, 
\ee
where the operators $\rH_{\sym}$ and $\rH_n$ in \eqref{eq:decompKHL} are Sturm--Liouville operators defined by the differential expression
\be\label{eq:tauA}
	\tau = -\frac{1}{\mu(t)}\frac{\rD}{\rD t}\mu(t)\frac{\rD}{\rD t},
\ee
however, in different $L^2$ spaces: $\rH_{\sym}$ acts in $L^2([0,\cL);\mu)$ on the domain
\[
\dom(\rH_{\sym})= \big\{f\in L^2([0, \cL); \mu)|\, f,\, \mu f' \in AC([0, \cL]),\ \tau f\in L^2([0, \cL); \mu);\ f'(0)=0\big\},
\]
and $\rH_n$, $n\ge 0$ are defined in  $L^2([t_n,\cL);\mu)$ on the domain
\[ 
\dom(\rH_n)= \big\{f\in L^2([t_n, \cL); \mu)|\, f,\, \mu f' \in AC([t_n, \cL]),\ \tau f\in L^2([t_n, \cL); \mu);\ f(t_n)=0\big\}.
\]
The weight function $\mu\colon [0, \mathcal L) \to [0, \infty)$ is explicitly given by
\begin{align}\label{eq:muL}
		 \mu(s)  = \sum_{n \geq 0} \mu_n \id_{[t_n, t_{n+1})}(s),\quad s\in [0,\cL).
\end{align}
By \eqref{eq:decompKHL}, $\ker(\bH)$ can be decomposed via the kernels of $\rH_{\sym}$ and $\rH_n$. Notice that $\ker(\rH_{\sym} )= \Span \{ \id_{[0, \cL)}\}$ and $\ker(\rH_n)= \Span \{ g_n\}$, where $g_n$ is given by
\[
	g_n(x) = \int_{t_n}^x \frac{1}{\mu(s)} \rD s, \qquad x \in[ t_n,\cL).
\]
With respect to the decomposition \eqref{eq:decompKHL}, every $g_n$ as well as $\id_{[0,\cL)}$ defines a function in $\ker(\bH)$. In particular, $\id_{[0,\cL)}$ gives rise to $\id_{\cT}$, which is clearly in $H^1(\cT)$. Since 
\[
\int_{t_n}^\infty |g_n'(x)|^2 \mu(x) \rD x = \int_{t_n}^\infty \frac{\rD x}{ \mu(x) } =     
 \sum_{k=n}^\infty  \frac{\ell_k}{ \mu_n} < \infty,
\]
according to \cite[Theorem 3.1]{ns01} (see also \cite[equation (3.12)]{sol04}), the other functions are also in $H^1(\cT)$.
Thus $\ker(\bH_G)$ is dense in $\ker(\bH)$, which completes the proof.
\end{proof}
\end{example}

\begin{remark}\label{rem:fampr}
Radially symmetric trees are a particular example of the so-called \emph{family preserving metric graphs} (see \cite{bl19} and also \cite{brke13}). 
 Employing the results from \cite{bl19} (and also assuming no horizontal edges), it is in fact possible to show that for family preserving metric graphs $\bH_0 = \bH_{G,\min}$ and hence either $\bH_G$ is closed (which holds exactly when the corresponding metric graph either has infinite volume, and hence $\bH_0$ is self-adjoint \cite[Remark 4.12]{kmn19}, or it has finite volume and finitely many ends) or its closure coincides with the maximal Kirchhoff Laplacian.
\end{remark}

%%%%%%%%%%%%%%%%%%%%%%%%%%%%%%%%%%%%%%%%%%%%%%%%%%%%%%%%%%%%%%%
%%%%%%%%%%%%%%%%%%%%%%%%%%%%%%%%%%%%%%%%%%%%%%%%%%%%%%%%%%%%%%%%%%%%%%%%%%%%%%%%%%%%%%%%%%%%%%%%%%%%%%%%%%%%%%%%%%%%%%%%%%%%%%

\noindent
\ack We thank Omid Amini and Matthias Keller for interesting discussions and Ivan Veseli\'c for bringing \cite{ebe} to our attention. 
We also thank the referee for the careful reading and remarks that have helped to improve the exposition.

N.N. appreciates the hospitality at the Centre de math\'ematiques Laurent Schwartz (\'Ecole Polytechnique) 
during a research stay funded by the OeAD (Marietta Blau-grant, ICM-2019-13386), where a part of this work was done.

\end{document}